\documentclass[12pt, reqno]{amsart}

\usepackage{amssymb, fullpage, amsmath, verbatim, mathtools}
\usepackage[margin=0.8in]{geometry}
\usepackage[colorlinks]{hyperref}
\hypersetup{citecolor=blue}
\usepackage[alphabetic]{amsrefs}
\numberwithin{equation}{section}
\theoremstyle{plain}
\newtheorem{thm}{Theorem}[section]

\newtheorem{prop}{Proposition}[section]

\theoremstyle{definition}

\newtheorem{remark}{Remark}[section]

\newcommand{\qede}{\hfill $\diamond$}

\newcommand{\R}{\mathbb R}

\newcommand{\N}{\mathbb N}

\newcommand{\lip}{\mathrm{lip}}

\newcommand{\diam}{\mathrm{diam}}
\renewcommand{\b}{\beta}
\newcommand{\e}{\mathbf e}
\renewcommand{\H}{\mathbb H}
\renewcommand{\S}{\mathbb S}

\begin{document}

\title{Ergodic behaviour of a Douglas-Rachford operator away from the origin}
\begin{abstract}
It is shown that away from the origin, the Douglas-Rachford operator with respect to a sphere and a convex set in a Hilbert space can be approximated by a another operator which satisfies a weak ergodic theorem. Similar results for other projection and reflection operators are also discussed.
\end{abstract}

\author{Jonathan M. Borwein}

\author{Ohad Giladi}

\address{Computer Assisted Research Mathematics and its Applications (CARMA), The University of Newcastle, University Drive, Callaghan NSW 2308, Australia}




\thanks{This research was supported by ARC grant DP160101537}


\maketitle

\section{Introduction}

\subsection{Background} 
Given a set $A$ in a Hilbert space $\H$, denote by $P_A:\H \rightrightarrows \H$ the multi-valued projection operator, that is,
\begin{align*}
P_Ax = \Big\{y\in A~\Big|~ \|x-y\| = \inf_{z\in A}\|x-z\|\Big\},
\end{align*}
where here and in what follows, $\|\cdot\|$ denotes the Hilbert norm on $\H$. Also, if $I:\H\to \H$ is the identity operator, denote by $R_A:\H \rightrightarrows \H$ the reflection operator, which is given by
\begin{align*}
R_A = 2P_A - I.
\end{align*}
Given two sets $A,B\subseteq \H$, define the Doulgas-Rachford operator by
\begin{align}\label{def DR}
T_{A,B} = \frac {I+R_BR_A}{2}.
\end{align}
Given $x\in \H$, let $\{x_n\}_{n=0}^{\infty}\subseteq \H$, be the sequence which is defined as follows,
\begin{align}\label{def iteration}
x_{n+1} = T_{A,B}x_n = T_{A,B}^nx_0, \quad x_0 = x.
\end{align}
This sequence is also known as the \emph{Douglas-Rachford iteration of $x$}. It was studied first in~\cite{DR56} as an algorithm for finding an intersection point of two sets. Indeed, it is not hard to check that 
\begin{align}\label{cond iff}
Tx = x \iff P_Ax \in A\cap B,
\end{align} 
and so any point $x\in A\cap B$ is a fixed point of $T_{A,B}$.


Analysing the Douglas-Rachford operator~\eqref{def DR} and the iteration sequence~\eqref{def iteration} are well known questions with interesting applications. This question has been studied in a convex setting (that is, when both $A$ and $B$ are convex), as well as in a non-convex setting (when either $A$ or $B$ is not convex). See for example~\cites{BCL 02, LM79} for the convex case and~\cites{ERT07, GE08} for the non-convex case.


In the case $A$ is convex, it is known that the projection operator $P_A$ is firmly non-expansive, that is, for every $x,y\in \H$,
\begin{align*}
\|P_Ax - P_Ay\|^2 + \|(I-P_A)x - (I-P_A)y\|^2 \le \|x-y\|^2.
\end{align*}
See for example~\cite{GK90}*{Thm. 12.2}. It then follows that the reflection operator $R_A$ is non-expansive, that is, for every $x,y\in \H$,
\begin{align*}
\|R_Ax-R_Ay\| \le \|x-y\|,
\end{align*}
and the Douglas-Rachford operator is firmly non-expansive. See for example~\cite{GK90}*{Thm. 12.1}. From the results of~\cite{Opi67}, it then follows that the Douglas-Rachford iteration~\eqref{def iteration} is weakly convergent. In the case $\H$ is finite dimensional, the weak convergence implies strong (norm) convergence.


While the convex case is well understood, much less is known about the non-convex case. One of the simplest examples of a non-convex setting is the case of a sphere and a line. This case was studied in~\cites{AB13, BS11, Ben15, Gil16}. Let 
\begin{align}\label{def sphere}
\S = \big\{x\in \H~|~\|x\|=1\big\},
\end{align}
and for $\lambda \ge 0$,
\begin{align}\label{def line}
L_\lambda = \big\{t\e_1+\lambda \e_2\in \H~|~ t\in \R \big\},
\end{align}
where here $\{\e_1,\e_2,\dots\}$ is an orthonormal basis of $\H$. It was shown in~\cite{Ben15} that if $\lambda \in (0,1)$, then for every $x\in \H$ with $\langle x,\e_1\rangle \neq 0$, the Douglas-Rachford iteration converges in norm to one of the two intersection points of $\S$ and $L_\lambda$. Here and in what follows $\langle \cdot, \cdot \rangle$ denotes the inner product on $\H$. Global convergence for the case $\lambda =0$ was already proved in~\cite{BS11}. The result in~\cite{Ben15} improved previous results, which only gave local convergence. It was also shown in~\cite{BS11}, that if $\langle x,\e_1\rangle = 0 $ or if $\lambda \ge 1$, the Douglas-Rachford iteration is not convergent. Note that the case $\lambda \le 0$ is completely analogous. Other non-convex cases were considered in~\cites{ABT16, HL13, Pha16}.


\subsection{An ergodic theorem for Lipschitz approximations of the Douglas-Rachford operator}\label{sec erg present}

It follows from the results of~\cite{Ben15}, that the convergence of the Douglas-Rachford iteration is~\emph{uniform} on compact sets. See~\cite{Gil16} for the exact argument (in~\cite{Gil16} one considers a finite dimensional Hilbert space, but the case for an infinite dimensional space is similar). Define the following sets,
\begin{align}\label{def H}
\H_+ = \big\{x\in \H~|~\langle x,\e_1\rangle >0\big\}, \quad \H_- = \big\{x\in \H~|~\langle x,\e_1\rangle <0\big\}, \quad \H_0 = \big\{x\in \H~|~\langle x,\e_1\rangle =0\big\}.
\end{align}
It is straightforward to show that if $T = T_{\S, L_\lambda}$, then $T(\H_+)\subseteq \H_+$, $T(\H_-)\subseteq \H_-$, $T(\H_0) \subseteq \H_0$. In particular, it follows that if $K\subseteq \H_+$ or $K\subseteq \H_-$ is compact, then
\begin{align}\label{weak ergodic}
\sup_{x,y\in K}\|T^nx - T^ny\| \stackrel{n\to \infty}{\longrightarrow} 0.
\end{align}
An estimate of the form~\eqref{weak ergodic} is also known as a \emph{weak ergodic theorem}. This type of theorems appears in the literature of population biology. See for example~\cite{Coh79}. See also~\cites{Nus90, RZ03} for further discussion on weak ergodic theorems.


In this note, we are interested in an estimate of the form~\eqref{weak ergodic} for the Douglas-Rachford operator in a more general setting where one of the sets is the unit sphere $\S$~\eqref{def sphere} and the other set is a convex set in $\H$, and the two sets have non-empty intersection (also known as the feasible case). This of course includes the case of the sphere and any affine subspace of $\H$. While we are unable to show an estimate of the form~\eqref{weak ergodic} for the Douglas-Rachford operator itself, what we can show is that away from the origin, the Douglas-Rachford operator can be approximated by another operator that satisfies~\eqref{weak ergodic}. The main result of this note reads as follows.

\begin{thm}\label{thm ergodic}
Assume that $C\subseteq \H$ is a convex set, let $\S$ be the unit sphere in $\H$~\eqref{def sphere}, and assume that $\S \cap C \neq \emptyset$. Let $T = T_{\S,C}$, and let $x_0\in \S\cap C$. Assume also that $\alpha,\b,r \ge 0$ are such that $\b\in [0,1)$, $r \ge \frac{2}{1-\b}$, and $\alpha \le \frac 1{1-\b}$. Then there exists $G:\H \to \H$ such that
\begin{align*}
\sup_{x\in B[x_0,r] \setminus B(0,1-\b)}\|Gx - Tx\| \le 2r\left(1-\alpha(1-\b)\right),
\end{align*}
and for all $n \in \N$,
\begin{align*}
\sup_{x,y \in B[x_0,r]}\|G^nx-G^ny\| \le 2r\alpha^n.
\end{align*}
\end{thm}

In Theorem~\ref{thm ergodic} and in what follows, $B(x,r)$ denotes the open ball around $x$ with radius $r$ with respect to the norm $\|\cdot\|$, while $B[x,r]$ denotes the closed ball. If we consider $T=T_{C,\S}$ rather than $T_{\S,C}$, Theorem~\ref{thm ergodic} does not necessarily hold. See Remark~\ref{exchange sets lip} and Remark~\ref{exchange sets ergodic} below.


The proof of Theorem~\ref{thm ergodic} is done in two steps. First, it is shown that away from the origin, the Douglas-Rachford operator satisfies a Lipschitz condition, and so using classical extension results, it can be extended to a Lipschitz map on all of $\H$. This is discussed in Section~\ref{sec lip}. By using further smoothing operations, it is shown that away from the origin, the Douglas-Rachford operator can be approximated by another operator which satisfies an estimate of the form~\eqref{weak ergodic}. The proof of Theorem~\ref{thm ergodic} is presented in Section~\ref{sec ergodic pf}.


In the special case where $C = L_\lambda$, as defined in~\eqref{def line}, we have in fact a slightly stronger result, namely that we can construct $G$ such that $\H_+\cup \H_0$ (alternatively, $\H_-\cup \H_0$) is invariant under $G$. See Remark~\ref{lip case line} and Remark~\ref{ergodic case line} below.


\subsection{Other projection and reflection operators}
Given two sets $A,B\subseteq \H$, the Douglas-Rachford operator~\eqref{def DR} is a special case of the following parametric family of operators. Given $s_1,s_2,s_3\in [0,1]$, define
\begin{align}\label{def family}
T_{A,B}^{s_1,s_2,s_3} = s_1I+(1-s_1)\left(s_2I+(1-s_2)R_B\right)\left(s_3I+(1-s_3)R_A\right).
\end{align}
As before, $I$ denotes the identity operator and $R_A$, $R_B$, denote the reflection operators on $A$, $B$, respecitively. Note that the Douglas-Rachford operator defined in~\eqref{def DR} corresponds to the case $s_1=\frac 1 2$, $s_2=s_3=0$. See~\cite{BST15} for a more detailed discussion on this family of operators. It is straightforward to show that the main result, Theorem~\ref{thm ergodic}, holds in fact for this more general family~\eqref{def family}. See Remark~\ref{family lip} and Remark~\ref{family ergodic} below. 

\begin{thm}\label{thm family}
Assume that $C\subseteq \H$ is a convex set, let $\S$ be the unit sphere in $\H$~\eqref{def sphere} and assume that $\S \cap C \neq \emptyset$. Let $s_1,s_2,s_3\in [0,1]$, let $T = T^{s_1,s_2,s_3}_{\S,C}$, and let $x_0\in \S\cap C$. Assume also that $\alpha,\b,r \ge 0$ are such that $\b\in [0,1)$, and $r$ and $\alpha$ satisfy
\begin{align*}
r \ge \frac{2(1+\b-2(s_1+(1-s_1)(s_2+s_3) + (1-s_1)(1-s_2)s_3)\b)}{1-\b},
\end{align*} 
and 
\begin{align*}
\alpha \le \frac{1+\b-2(s_1+(1-s_1)(s_2+s_3) + (1-s_1)(1-s_2)s_3)\b}{1-\b}.
\end{align*} 
Then there exists $G:\H\to \H$ such that
\begin{align*}
\sup_{x\in B[x_0,r] \setminus B(0,1-\b)}\|Gx - Tx\| \le 2r\left(1-\frac{\alpha(1-\b)}{1+\b-2(s_1+(1-s_1)(s_2+s_3) + (1-s_1)(1-s_2)s_3)\b}\right),
\end{align*}
and for all $n \in \N$,
\begin{align*}
\sup_{x,y \in B[x_0,r]}\|G^nx-G^ny\| \le 2r\alpha^n.
\end{align*}
\end{thm}


Note that choosing $s_1 = \frac 1 2 $ and $s_2=s_3 =0$ in Theorem~\ref{thm family} gives Theorem~\ref{thm ergodic}. Another well known case is when $s_1=0$ and $s_2=s_3 = \frac 1 2$, in which case we obtain
\begin{align*}
T_{A,B}^{0,\frac 1 2 , \frac 1 2} = P_BP_A,
\end{align*}
also known as the Von-Neuman operator~\cite{VN50}. Regarding the convergence of the iteration sequence $x_{n+1} = P_BP_Ax_n$, $x_0=x$, it was shown in~\cite{VN50} that if $A$, $B$, are both subspaces in $\H$, then $x_n\stackrel{n\to \infty}{\longrightarrow} P_{A\cap B}x$ (norm convergence). It was later shown in~\cite{BB93} that if $0\in \mathrm{int}(A-B)$ or $A-B$ is a closed subspace, then the iteration sequence converges linearly (that is, when the rate of convergence is $c\alpha^n$, where $c>0$ is a constant and $\alpha \in [0,1)$). 


For the von Neumann operator, we have in fact a stronger result than Theorem~\ref{thm ergodic}, which reads as follows.


\begin{thm}\label{thm VN}
Assume that $C\subseteq \H$ is a convex set, let $\S$ be the unit sphere in $\H$~\eqref{def sphere}, and assume that $\S\cap C \neq \emptyset$. Let $T = P_CP_\S$, and let $x_0 \in \S\cap C$. Also, assume that $\alpha,\b,r\ge 0$ are such that $\b\in [0,1)$, $r \ge 2$, and $\alpha \le \frac 1 {1-\b}$. Then there exists $G:\H \to \H$ such that
\begin{align*}
\sup_{x\in B[x_0,r]}\|Gx-Tx\| \le 2r\big(1-\alpha(1-\b)\big),
\end{align*}
and 
\begin{align*}
\sup_{x,y\in B[x_0,r]}\|G^nx-G^ny\| \le 2r\alpha^n.
\end{align*}
\end{thm}

Note that Theorem~\ref{thm VN} is slightly stronger than Theorem~\ref{thm ergodic} since we only require $r \ge 2$, rather than $r \ge \frac{1}{1-\b}$. Similar to the case of Theorem~\ref{thm ergodic}, we cannot change the order of the projections in Theorem~\ref{thm VN}. See Remark~\ref{rem exchange VN} below. Theorem~\ref{thm VN} is proved in Section~\ref{sec family}.


\section{Lipschitz behaviour of the Douglas-Rachford operator}\label{sec lip}
 
Given two Banach spaces $(X,\|\cdot\|_X)$ and $(Y,\|\cdot\|_Y)$, a set $D\subseteq X$, and a map $f: D\to Y$, define the Lipschitz constant of $f$ to be
\begin{align*}
\|f\|_{\lip} = \sup_{\substack{x,y \in D \\ x\neq y}}\frac{\|f(x)-f(y)\|_Y}{\|x-y\|_X}.
\end{align*}
A map $f: X\to Y$ is said to be Lipschitz if $\|f\|_\lip < \infty$. Note that if $C\subseteq \H$, then $T_{\S,C}$ is not necessarily Lipschitz on $\H$, since $P_\S = x/\|x\|$, which is not Lipschitz. However, it is shown below that if $C\subseteq \H$ is convex, the Douglas-Rachford operator can be `smoothed' in a neighbourhood of the origin such that the smoothed operator satisfies a Lipschitz condition.


\begin{thm}\label{thm lip approx}
Assume that $C\subseteq \H$ is a convex set, and let $\S$ be the unit sphere in $\H$~\eqref{def sphere}. Let $T = T_{\S,C}$, and let $\b\in [0,1)$. Then there exists $F:\H \to \H$ such that 
\begin{align*}
F\big|_{\H \setminus B(0,1-\b)} = T,
\end{align*}
and
\begin{align*}
\|F\|_\lip \le \frac{1}{1-\b}.
\end{align*}
\end{thm}


We begin with the following proposition.


\begin{prop}\label{prop ref lip}
Assume that $x,y\in \H\setminus B(0,1-b)$. Then
\begin{align*}
\|R_\S x-R_\S y\| \le \frac{1+ \b}{1-\b}\|x-y\|.
\end{align*}
\end{prop}

\begin{proof}
Recall that 
\begin{align*}R_\S x = 2P_\S x-x =  \left(\frac 2 {\|x\|}-1\right)x.
\end{align*} 
Hence,
\begin{align}\label{bound norm R}
\nonumber \|R_\S x-R_\S y\|^2 & = \|R_\S x\|^2 + \|R_\S y\|^2 - 2\langle R_\S x,R_\S y\rangle 
\\
\nonumber & = (2-\|x\|)^2+(2-\|y\|)^2 - 2\left(\frac 2{\|x\|}-1\right)\left(\frac 2{\|y\|}-1\right)\langle x,y\rangle 
\\ 
\nonumber & = 4-4\|x\|+\|x\|^2+4-4\|y\|+\|y\|^2-2\left(\frac 4{\|x\|\|y\|}-\frac 2{\|x\|}-\frac 2 {\|y\|}\right)\langle x,y\rangle -2\langle x,y\rangle
\\
\nonumber & = \|x\|^2+\|y\|^2-2\langle x,y\rangle + 8-4\|x\|-4\|y\| - \frac{2}{\|x\|\|y\|}\big(4 - 2\|x\| - 2\|y\|\big)\langle x,y\rangle 
\\
& = \|x-y\|^2 + 4\big(2-\|x\|-\|y\|\big)\left(1-\frac{\langle x,y\rangle}{\|x\|\|y\|}\right)
\end{align}
Now, since $\|x\|\|y\| \le \frac{\|x\|^2+\|y\|^2}{2}$ for all $x,y\in \H$, if $x,y \in \H\setminus B(0,1-\b)$, then
\begin{eqnarray*}
\|x\|\|y\| - \langle x,y \rangle & \le & \frac{\|x\|^2+\|y\|^2}{2} - \langle x,y \rangle 
\\
& \stackrel{(*)}{\le} & \frac{\|x\|\|y\|}{(1-\b)^2}\left(\frac{\|x\|^2+\|y\|^2}{2} - \langle x,y \rangle\right)
\\ 
& = & \frac{\|x\|\|y\|}{2(1-\b)^2}\left(\|x\|^2+\|y\|^2 - 2\langle x,y \rangle\right) 
\\
& = & \frac{\|x\|\|y\|}{2(1-\b)^2}\|x-y\|^2,
\end{eqnarray*}
where in ($*$) we used the fact that 
\begin{align*}
\frac{\|x\|^2+\|y\|^2}{2} - \langle x,y \rangle \ge \|x\|\|y\| - \langle x,y\rangle \ge 0,
\end{align*} 
and the fact that $\|x\| \ge 1-\b$ and $\|y\| \ge 1-\b$. Therefore, if $x,y\in \H\setminus B(0,1-\b)$, then
\begin{align}\label{bound one minus}
1-\frac{\langle x,y\rangle}{\|x\|\|y\|} \le \frac{1}{2(1-\b)^2}\|x-y\|^2.
\end{align}
Plugging~\eqref{bound one minus} into~\eqref{bound norm R}, it follows that if $x,y\in \H\setminus B(0,1-\b)$, then $2-\|x\|-\|y\| \le 2\b$, and so
\begin{align*}
\|R_\S x-R_\S y\|^2 & \le \left(1+ 4\big(2-\|x\|-\|y\|\big)\frac{1}{2(1-\b)^2}\right)\|x-y\|^2 
\\ & \le \left(1+ \frac {4\b} {(1-\b)^2}\right)\|x-y\|^2
\\ & = \frac{(1+\b)^2}{(1-\b)^2}\|x-y\|^2.
\end{align*}
Hence,
\begin{align*}
\|R_\S x-R_\S y\| \le \frac{1+\b}{1-\b}\|x-y\|,
\end{align*}
and this completes the proof.
\end{proof}

Another tool which is needed in the proof of Theorem~\ref{thm lip approx} is the following theorem, known as Kirszbraun's Theorem. See for example~\cites{BL00, GK90}. Given a set $D\subseteq \H$, let $\overline{\mathrm{conv}(D)}$ denote its closed convex hull, where the convex hull is given by
\begin{align*}
\mathrm{conv}(D) = \left\{\sum_{i=1}^n t_i x_i~\Big|~ x_i \in D, ~~t_i\ge 0, ~~1 \le i \le n, ~~ \sum_{i=1}^n t_i = 1, ~~ n \in \N\right\}.
\end{align*}

Kirsbraun's theorem reads as follows.

\begin{thm}\label{thm Kirs}
Assume that $D_1,D_2 \subseteq \H$. Assume that $f:D_1 \to D_2$ is Lipschitz. Then there exists $F: \H \to \overline{\mathrm{conv}(D_2)}$ such that $F\big|_{D_1} = f$ and $\|F\|_{\lip} = \|f\|_{\lip}$.
\end{thm}

We are now in a position to prove Theorem~\ref{thm lip approx}

\begin{proof}[Proof of Theorem~\ref{thm lip approx}]
Since $C$ is convex, it follows that $R_C$ is non-expansive. Let $x,y \in \H\setminus B(0,1-\b)$. Then
\begin{eqnarray*}
\|Tx-Ty\| & = & \left\|\left(\frac{I+R_CR_\S }{2}\right)x-\left(\frac{I+R_CR_\S }{2}\right)y\right\| 
\\ &=& \left\|\frac{x-y}{2} + \frac{R_CR_\S x-R_CR_\S y}{2}\right\|
\\ & \le & \frac 1 2 \|x-y\| + \frac 1 2 \|R_CR_\S x-R_CR_\S y\|
\\ & \stackrel{(*)}{\le} & \frac 1 2 \|x-y\|+ \frac 1 2 \|R_\S x-R_\S y\|
\\ & \stackrel{(**)}{\le} & \frac 1 2 \|x-y\| + \frac {1+\b}{2(1-\b)}\|x-y\| 
\\ & = & \frac{\|x-y\|}{1-\b},
\end{eqnarray*}
where in ($*$) we used the fact that $C$ is convex and thus $R_C$ is non-expansive, and in ($**$) we used Proposition~\ref{prop ref lip}. Applying Theorem~\ref{thm Kirs} to $T$ on the sets $ D_1 = \H\setminus B(0,1-\b)$ and $D_2 = \H$ completes the proof.
\end{proof}


\begin{remark}\label{family lip}
Note that if $T = T^{s_1,s_2,s_3}_{\S,C}$ is as defined in~\eqref{def family}, then in particular, 
\begin{align*}
T = (s_1+(1-s_1)(s_2+s_3))I + (1-s_1)s_2(1-s_3)R_\S + (1-s_1)(1-s_2)s_3R_C+(1-s_1)s_2s_3R_CR_\S.
\end{align*}
Note also that 
\begin{align*}
(s_1+(1-s_1)(s_2+s_3)) + (1-s_1)s_2(1-s_3) + (1-s_1)(1-s_2)s_3 + (1-s_1)s_2s_3 =1.
\end{align*}
Hence, if $C$ is convex, then since both $I$ and $R_C$ are non-expansive, using Proposition~\ref{prop ref lip}, for every $x,y\in \H\setminus B(0,1-\b)$,
\begin{align}\label{lip const family}
\nonumber \|Tx-Ty\| & \le (s_1+(1-s_1)(s_2+s_3) + (1-s_1)(1-s_2)s_3)\|x-y\|
\\ \nonumber & \quad + ((1-s_1)s_2(1-s_3)+(1-s_1)s_2s_3)\frac{1+\b}{1-\b}\|x-y\|
\\ & = \frac{1+\b-2(s_1+(1-s_1)(s_2+s_3) + (1-s_1)(1-s_2)s_3)\b}{1-\b}\|x-y\|.
\end{align}
Thus, repeating the proof of Theorem~\ref{thm lip approx}, we obtain a similar result, but now the Lipschitz constant is the one given in~\eqref{lip const family}. \qede
\end{remark}


\begin{remark}\label{exchange sets lip}
Even if $C\subseteq \H$ is convex, the map $x\mapsto R_\S R_Cx$ need not satisfy a Lipschitz condition, since $R_C$ might be arbitrarily close to $0$ (indeed, it might even not be defined). Thus, in general, Theorem~\ref{thm lip approx} does not hold for the operator $T=T_{C,\S}$. \qede
\end{remark}


\begin{remark}\label{lip case line}
In the case $C = L_\lambda$, as defined in~\eqref{def line}, if $T = T_{\S, L_\lambda}$, then $\H_+$, $\H_-$, $\H_0$ as defined in~\eqref{def H} are all invariant under $T$. Hence, by applying Theorem~\ref{thm Kirs} with $D_1 = (\H_+\cup \H_0)\setminus B(0,1-\b)$ (resp. $(\H_-\cup \H_0)\setminus B(0,1-\b)$) and $D_2 = \H_+$ (resp. $\H_-$), it follows that in Theorem~\ref{thm lip approx} we can choose $F:\H_+\cup \H_0 \to \H_+\cup\H_0$ (resp. $F:\H_-\cup \H_0 \to \H_-\cup \H_0$). Note that we cannot choose $F:\H_+ \to \H_+$ or $F:\H_- \to \H_-$ as these are not closed sets. \qede
\end{remark}


\section{Proof of Theorem~\ref{thm ergodic}}\label{sec ergodic pf}

Given a set $D\subseteq \H$, define
\begin{align*}
\diam(D) = \sup_{x,y\in D}\|x-y\|.
\end{align*}

The next proposition shows that on a bounded convex set, we can `smooth' Lipschitz maps, so that the smoothed map satisfies an estimate of the form~\eqref{weak ergodic}. The smoothing operation is similar to the one which appeared in~\cite{RZ03}.

\begin{prop}\label{prop smooth}
Assume that $D\subseteq \H$ is bounded and convex, and let $F:D \to D$ be a Lipschitz map. Then for every $\alpha \le \|F\|_\lip$ there exists a map $G:D \to D$ such that
\begin{align*}
\|Gx-Fx\| \le \left(1-\frac \alpha{\|F\|_\lip}\right)\diam(D),
\end{align*}
and for all $n\in \N$,
\begin{align*}
\sup_{x,y\in D}\|G^nx-G^ny\| \le \alpha^n\diam(D).
\end{align*}
In particular, if $\alpha \in [0,1)$,
\begin{align*}
\sup_{x,y\in D}\|G^nx-G^ny\| \stackrel{n\to \infty}{\longrightarrow} 0.
\end{align*}
\end{prop}

\begin{proof}
Let $\theta \in D$ and $\gamma\in [0,1]$. Define
\begin{align*}
Gx = (1-\gamma)Fx + \gamma \theta.
\end{align*}
Then since $D$ is convex, it follows that $G(D)\subseteq D$, and 
\begin{align*}
\sup_{x\in D}\|Gx-Fx\| = \sup_{x\in D}\gamma\|Fx-\theta\| \le \gamma\,\diam(D).
\end{align*}
Also, 
\begin{align*}
\|G\|_\lip = (1-\gamma)\|F\|_\lip.
\end{align*}
Choosing $\gamma = 1-\frac \alpha{\|F\|_\lip} \in [0,1]$ and using the fact that $\|G^n\|_\lip \le \|G\|_\lip^n$ completes the proof.
\end{proof}


We are now in a position to prove Theorem~\ref{thm ergodic}.


\begin{proof}[Proof of Theorem~\ref{thm ergodic}]
Since $x_0\in \S\cap C$, we have $Tx_0 = x_0$, see~\eqref{cond iff}. Let $F:\H \to \H$ be the map obtained from Theorem~\ref{thm lip approx}. Let $x\in B[x_0,r]$. If $x\notin B[0,1]$ then $R_\S = R_{\mathbb B}$, where
\begin{align*}
\mathbb B = \big\{x\in \H~|~ \|x\| \le 1\big\},
\end{align*}
which is convex. Thus, in this case, $R_C$, $R_\S$ and therefore $T$ are all non-expansive, and so
\begin{align*}
\|Fx-Fx_0\| =\|Tx-Tx_0\| \le \|x-x_0\| \le r.
\end{align*}
If $x\in B[0,1]$, then by Theorem~\ref{thm lip approx},
\begin{align*}
\|Fx - Fx_0\| \le \frac{\|x-x_0\|}{1-\b} \le \frac{2}{1-\b}.
\end{align*}
Therefore, if $r \ge \frac{2}{1-\b}$, then 
\begin{align*}
F(B[x_0,r]) \subseteq B[x_0,r].
\end{align*}
Now, $\diam(B[x_0,r]) = 2r$. Applying Proposition~\ref{prop smooth} to the function $F$ on the domain $D=B[x_0,r]$, it follows that for every $\alpha \le \frac1{1-\b}$, there exists $G:\H\to \H$ which satisfies $G(B[x_0,r])\subseteq B[x_0,r]$, and such that 
\begin{align*}
\sup_{x\in B[x_0,r]}\|Gx-Fx\| \le 2r\left(1- \alpha(1-\b)\right),
\end{align*}
and 
\begin{align*}
\sup_{x,y\in B[x_0,r]}\|G^nx-G^ny\| \le 2r\alpha^n.
\end{align*}
Since 
\begin{align*}
\sup_{x\in B[x_0,r]\setminus B(0,1-\b)}\|Gx-Tx\| \le \sup_{x\in B[x_0,r]}\|Gx-Fx\|,
\end{align*}
the proof is complete.
\end{proof}


\begin{remark}
Note that by Proposition~\ref{prop smooth}, the choice of $G$ in Theorem~\ref{thm ergodic} depends on $\alpha$ and on the centre point $x_0$. \qede
\end{remark}


\begin{remark}\label{family ergodic}
If we consider now the operator $T = T_{\S,C}^{s_1,s_2,s_3}$ as defined in~\eqref{def family}, then repeating the proof of Theorem~\ref{thm ergodic} but now using Remark~\ref{family lip}, we obtain Theorem~\ref{thm family}. Note that the conditions on $\alpha$ and $r$ that we need are $r \ge 2\|F\|_\lip$ and $\alpha \le \|F\|_{\lip}$, where $F$ is the function obtained in Theorem~\ref{thm lip approx} (applied now to the operator $T$). These are exactly the conditions that appear in Theorem~\ref{thm family}.  \qede
\end{remark}


\begin{remark}\label{exchange sets ergodic}
Since, by Remark~\ref{exchange sets lip}, Theorem~\ref{thm lip approx} does not necessarily hold if we let $T = T_{C,\S}$, the same is true for Theorem~\ref{thm ergodic}. \qede
\end{remark}


\begin{remark}\label{ergodic case line}
In the case of the sphere and the line, $C = L_\lambda$, $\lambda \in [0,1]$ as defined in~\eqref{def line}, it follows from Remark~\ref{lip case line} that we can choose $G:\H_+\cup\H_0 \to \H_+\cup \H_0$ such that 
\begin{align*}
\sup_{\substack{ x\in B[x_0,r]\setminus B(0,1-\b) \\ x\in \H_+\cup \H_0}}\|Gx - Tx\| \le 2r\left(1-\alpha(1-\b)\right),
\end{align*}
and for all $n\in \N$,
\begin{align*}
\sup_{\substack{x,y \in B[x_0,r] \\ x,y\in \H_+\cup \H_0}}\|G^nx-G^ny\| \le 2r\alpha^n.
\end{align*}
If we replace $\H_+$ by $\H_-$ we obtain a similar result. \qede
\end{remark}


\section{Proof of Theorem~\ref{thm VN}}\label{sec family}

We begin with the following proposition, which shows that the projection operator on the sphere, $P_\S$, satisfies a Lipschitz condition away from the origin.

\begin{prop}\label{prop proj lip}
For every $x,y\in \H\setminus\{0\}$,
\begin{align*}
\left\|\frac x {\|x\|} - \frac y {\|y\|}\right\| \le \max\left\{\frac 1 {\|x\|}, \frac 1 {\|y\|}\right\}\|x-y\|.
\end{align*}
In particular, if $\b\in[0,1)$, $x,y \in \H\setminus B(0,1-\b)$, and $\S$ is the unit sphere in $\H$~\eqref{def sphere}, 
\begin{align*}
\|P_\S x-P_\S y\| \le \frac{\|x-y\|}{1-\b}.
\end{align*}
\end{prop}

\begin{proof}
Assume without loss of generality that $\|x\| \le \|y\|$. Then
\begin{eqnarray*}
&& \frac{1}{\|x\|^2}\|x-y\|^2 - \left\|\frac x{\|x\|}-\frac y {\|y\|}\right\|^2  = \frac{\|y\|^2}{\|x\|^2} -2\langle x,y\rangle \left(\frac 1{\|x\|^2}-\frac 1 {\|x\|\|y\|}\right)-1
\\
&& \quad\stackrel{(*)}{\ge}\frac{\|y\|^2}{\|x\|^2} -2\|x\|\|y\|\left(\frac 1{\|x\|^2}-\frac 1 {\|x\|\|y\|}\right)-1 = \frac{\|y\|^2}{\|x\|^2} - 2\frac{\|y\|}{\|x\|}+1 = \left(\frac{\|y\|}{\|x\|}-1\right)^2 \ge 0,
\end{eqnarray*}
where in ($*$) we used the fact that $\langle x,y \rangle \le \|x\|\|y\|$ and the fact that $\frac{1}{\|x\|^2}-\frac 1{\|x\|\|y\|} \ge 0$ (since $\|x\| \le \|y\|$). Thus,
\begin{align*}
\left\|\frac x{\|x\|}-\frac y {\|y\|}\right\| \le  \frac{1}{\|x\|}\|x-y\|,
\end{align*}
which completes the proof of the first statement. The second statement follows as $P_\S x = x/\|x\|$ for all $x\in \H\setminus\{0\}$.
\end{proof}


We are now in a position to prove Theorem~\ref{thm VN}.


\begin{proof}[Proof of Theorem~\ref{thm VN}]
Note first that if $r \ge 2$, then since $x_0\in \S$, $B[0,1] \subseteq B[x_0,r]$. Therefore, $P_\S(B[x_0,r]\setminus\{0\}) = \S$. Now, since $C$ is convex, $P_C$ is non-expansive, and so for all $x\in B[x_0,r]\setminus \{0\}$,
\begin{align}\label{dist less 2}
\|P_CP_\S x - P_CP_\S x_0\| \le \|P_\S x - P_\S x_0\| \le 2 \le r.
\end{align}
Therefore,
\begin{align}\label{inside ball}
P_CP_\S(B[x_0,r]\setminus \{0\}) \subseteq B[x_0,r]
\end{align}
In particular, it follows that 
\begin{align*}
P_CP_\S((1-\b)\S) \subseteq B[x_0,r],
\end{align*}
where 
\begin{align*}
(1-\b)\S = \big\{x\in \H~|~ \|x\| = 1-\b\big\}.
\end{align*}
Thus, by Theorem~\ref{thm Kirs}, there exists $F:\H \to B[x_0,r]$ such that $\|F\|_\lip = \frac 1 {1-\b}$ and $F\big|_{(1-\b)\S} = P_CP_\S$. Define, $F_1:\H\to \H$,
\begin{align}\label{def F1}
F_1x =\begin{dcases} Fx & x\in B[0,1-\b] , \\ P_CP_\S x & x\in \H\setminus B(0,1-\b). \end{dcases}
\end{align}
If $x,y \in B[0,1-\b]$ or $x,y\in \H\setminus B(0,1-\b)$ then since $\|F\|_\lip = \frac 1 {1-\b}$ and by Proposition~\ref{prop proj lip},
\begin{align}\label{case both}
\|F_1x-F_1y\| \le \frac{\|x-y\|}{1-\b}.
\end{align} 
If, without loss of generality, $x\in B[0,1-\b]$ and $y\in \H\setminus B(0,1-\b)$, then there exists $t\in [0,1]$ such that $\|tx+(1-t)y\| = 1-\b$. Thus,
\begin{eqnarray}\label{case one}
\nonumber \|F_1x-F_1y\| & \le & \|F_1x-F_1(tx+(1-t)y)\| + \|F_1(tx+(1-t)y) - F_1y\| 
\\ \nonumber & \stackrel{\eqref{def F1}}{=} & \|Fx-F(tx+(1-t)y)\| + \|P_CP_\S (tx+(1-t)y) - P_CP_\S y\| 
\\ \nonumber & \stackrel{(*)}{\le} & (1-t)\frac{\|x-y\|}{1-\b}+ t\frac{\|x-y\|}{1-\b} 
\\ & = & \frac{\|x-y\|}{1-\b},
\end{eqnarray}
where in ($*$) we used the fact that $\|F\|_\lip = \frac 1{1-\b}$ and Proposition~\ref{prop proj lip}. Combining~\eqref{case both} and~\eqref{case one}, it follows that $\|F_1\|_\lip = \frac 1 {1-\b}$. Now, if $r\ge 2$, 
\begin{align*}
F_1(B[0,1-\b]) = F(B[0,1-\b]) \stackrel{(*)}{\subseteq} B[x_0,r],
\end{align*}
and 
\begin{align*}
F_1(B[x_0,r]\setminus B(0,1-\b)) = P_CP_\S(B[x_0,r]\setminus B(0,1-\b)) \stackrel{\eqref{inside ball}}{\subseteq} B[x_0,r],
\end{align*}
where in ($*$) we used the fact that $F(\H)\subseteq  B[x_0,r]$. Altogether,
\begin{align*}
F_1(B[x_0,r]) \subseteq  B[x_0,r],
\end{align*}
and $\|F_1\|_\lip = \frac 1 {1-\b}$. Applying Proposition~\ref{prop smooth} to $F_1$ on the domain $B[x_0,r]$ completes the proof.
\end{proof}


\begin{remark}\label{rem exchange VN}
Note that we cannot change the order of projections in Theorem~\ref{thm VN}. Indeed, it is possible that $P_Cx=0$ for some $x\in \H$, and then $P_\S P_Cx$ is not defined. Even if $\|P_Cx\| >0$, $\|P_Cy\| >0$, then by Proposition~\ref{prop proj lip}, 
\begin{align*}
\|P_\S P_C x - P_\S P_C y\| \le \max\left\{\frac 1{\|P_Cx\|}, \frac 1 {\|P_Cy\|}\right\}\|P_Cx - P_Cy\| \le \max\left\{\frac 1{\|P_Cx\|}, \frac 1 {\|P_Cy\|}\right\}\|x - y\|,
\end{align*}
but $\max\left\{\frac 1{\|P_Cx\|}, \frac 1 {\|P_Cy\|}\right\}$ can be very large. Thus, we do not obtain an estimate similar to~\eqref{dist less 2}. \qede
\end{remark}


\subsection*{Acknowledgements} This note is a revised and much simplified version of a note whose original version can be found at~\url{https://www.carma.newcastle.edu.au/jon/weak-ergodicity.pdf}. Note that the results in the original version apply only for the case of the sphere and a line in finite dimensional spaces, while here the results are more general. 
Sadly, the first named author passed away before this note was being revised. The second named author is grateful to Jon Borwein for many interesting conversations and for his warm friendship.


\begin{thebibliography}{99}

\bib{AB13}{article}{
   author={Arag{\'o}n Artacho, F. J.},
   author={Borwein, J. M.},
   title={Global convergence of a non-convex Douglas-Rachford iteration},
   journal={J. Global Optim.},
   volume={57},
   date={2013},
   number={3},
   pages={753--769},
   issn={0925-5001},
}

\bib{ABT16}{article}{
   author={Arag{\'o}n Artacho, F. J.},
   author={Borwein, J. M.},
   author={Tam, M. K.},
   title={Global behavior of the Douglas-Rachford method for a nonconvex
   feasibility problem},
   journal={J. Global Optim.},
   volume={65},
   date={2016},
   number={2},
   pages={309--327},
   issn={0925-5001},
}

\bib{BB93}{article}{
   author={Bauschke, H. H.},
   author={Borwein, J. M.},
   title={On the convergence of von Neumann's alternating projection
   algorithm for two sets},
   journal={Set-Valued Anal.},
   volume={1},
   date={1993},
   number={2},
   pages={185--212},
   issn={0927-6947},
}

\bib{BCL02}{article}{
   author={Bauschke, H. H.},
   author={Combettes, P. L.},
   author={Luke, D. R.},
   title={Phase retrieval, error reduction algorithm, and Fienup variants: a
   view from convex optimization},
   journal={J. Opt. Soc. Amer. A},
   volume={19},
   date={2002},
   number={7},
   pages={1334--1345},
   issn={1084-7529},
}


\bib{Ben15}{article}{
    author={Benoist, J.},
    title={The Douglas--Rachford algorithm for the case of the sphere and the line},
    journal={J. Global Optim.},
    year={2015},
    volume={63},
    number={2},
    pages={363--380},
    issn={1573-2916},
}

\bib{BL00}{book}{
   author={Benyamini, Y.},
   author={Lindenstrauss, J.},
   title={Geometric nonlinear functional analysis. Vol. 1},
   series={American Mathematical Society Colloquium Publications},
   volume={48},
   publisher={American Mathematical Society, Providence, RI},
   date={2000},
   pages={xii+488},
   isbn={0-8218-0835-4},
}

\bib{BS11}{article}{
   author={Borwein, J. M.},
   author={Sims, B.},
   title={The Douglas-Rachford algorithm in the absence of convexity},
   conference={
      title={Fixed-point algorithms for inverse problems in science and
      engineering},
   },
   book={
      series={Springer Optim. Appl.},
      volume={49},
      publisher={Springer, New York},
   },
   date={2011},
   pages={93--109},
}

\bib{BST15}{article}{
   author={Borwein, J. M.},
   author={Sims, B.},
   author={Tam, M. K.},
   title={Norm convergence of realistic projection and reflection methods},
   journal={Optimization},
   volume={64},
   date={2015},
   number={1},
   pages={161--178},
   issn={0233-1934},
}

\bib{Coh79}{article}{
   author={Cohen, J. E.},
   title={Ergodic theorems in demography},
   journal={Bull. Amer. Math. Soc. (N.S.)},
   volume={1},
   date={1979},
   number={2},
   pages={275--295},
   issn={0273-0979},
}

\bib{DR56}{article}{
   author={Douglas, J., Jr.},
   author={Rachford, H. H., Jr.},
   title={On the numerical solution of heat conduction problems in two and
   three space variables},
   journal={Trans. Amer. Math. Soc.},
   volume={82},
   date={1956},
   pages={421--439},
   issn={0002-9947},
}


\bib{ERT07}{article}{
   author={Elser, V.},
   author={Rankenburg, I.},
   author={Thibault, P.},
   title={Searching with iterated maps},
   journal={Proc. Natl. Acad. Sci. USA},
   volume={104},
   date={2007},
   number={2},
   pages={418--423 (electronic)},
   issn={1091-6490},
}

\bib{Gil16}{article}{
	author={Giladi, O.},
	title={A remark on the convergence of the Douglas-Rachford iteration in a non-convex setting},
	year={2016},
	note={Preprint available at~\url{https://sites.google.com/site/adfgh1469/publications}},
}


\bib{GK90}{book}{
   author={Goebel, K.},
   author={Kirk, W. A.},
   title={Topics in metric fixed point theory},
   series={Cambridge Studies in Advanced Mathematics},
   volume={28},
   publisher={Cambridge University Press, Cambridge},
   date={1990},
   pages={viii+244},
   isbn={0-521-38289-0},
}

\bib{GE08}{article}{
  title = {Divide and concur: A general approach to constraint satisfaction},
  author = {Gravel, S.},
  author = {Elser, V.},
  journal = {Phys. Rev. E},
  volume = {78},
  number = {3},
  pages = {036706},
  year = {2008},
  publisher = {American Physical Society},
}

\bib{HL13}{article}{
   author={Hesse, R.},
   author={Luke, D. R.},
   title={Nonconvex notions of regularity and convergence of fundamental
   algorithms for feasibility problems},
   journal={SIAM J. Optim.},
   volume={23},
   date={2013},
   number={4},
   pages={2397--2419},
   issn={1052-6234},
}

\bib{LM79}{article}{
   author={Lions, P.-L.},
   author={Mercier, B.},
   title={Splitting algorithms for the sum of two nonlinear operators},
   journal={SIAM J. Numer. Anal.},
   volume={16},
   date={1979},
   number={6},
   pages={964--979},
   issn={0036-1429},
}

\bib{VN50}{book}{
   author={von Neumann, J.},
   title={Functional Operators. II. The Geometry of Orthogonal Spaces},
   series={Annals of Mathematics Studies, no. 22},
   publisher={Princeton University Press, Princeton, N. J.},
   date={1950},
   pages={iii+107},
}

\bib{Nus90}{article}{
   author={Nussbaum, R. D.},
   title={Some nonlinear weak ergodic theorems},
   journal={SIAM J. Math. Anal.},
   volume={21},
   date={1990},
   number={2},
   pages={436--460},
   issn={0036-1410},
}

\bib{Opi67}{article}{
   author={Opial, Z.},
   title={Weak convergence of the sequence of successive approximations for
   nonexpansive mappings},
   journal={Bull. Amer. Math. Soc.},
   volume={73},
   date={1967},
   pages={591--597},
   issn={0002-9904},
}

\bib{Pha16}{article}{
   author={Phan, H. M.},
   title={Linear convergence of the Douglas-Rachford method for two closed
   sets},
   journal={Optimization},
   volume={65},
   date={2016},
   number={2},
   pages={369--385},
   issn={0233-1934},
}
		

\bib{RZ03}{article}{
   author={Reich, S.},
   author={Zaslavski, A.},
   title={A weak ergodic theorem for infinite products of Lipschitzian
   mappings},
   journal={Abstr. Appl. Anal.},
   date={2003},
   number={2},
   pages={67--74},
   issn={1085-3375},
}

\end{thebibliography}
\end{document}